\documentclass[12pt,a4paper]{amsart}
\usepackage{amsfonts,amssymb,amscd,amsmath,enumerate,verbatim}
\usepackage{amsmath,amsthm,amssymb}
\usepackage{ulem,ascmac} 
\usepackage[dvips]{graphicx}
\usepackage{fancybox}

\usepackage{url}

\theoremstyle{plain}
\newtheorem{thm}{Theorem}[section]
\newtheorem{prop}[thm]{Proposition}
\newtheorem{cor}[thm]{Corollary}
\newtheorem{lem}[thm]{Lemma}
\newtheorem{conj}[thm]{Conjecture}

\theoremstyle{definition}

\newtheorem{rem}[thm]{Remark}
\newtheorem{prob}[thm]{Problem}
\newtheorem{obs}[thm]{Observation}
\newtheorem{cl}[thm]{Claim}

\newcommand{\Tc}{{\mathcal{T}}}

\newcommand{\rleft}{\mathopen{}\mathclose\bgroup\left}
\newcommand{\rright}{\aftergroup\egroup\right}

\newcommand{\C}{{\mathcal C}}
\newcommand{\Kc}{{\rm Kc}}
\usepackage{color}

\title[Kempe equivalence of almost bipartite graphs]{Kempe equivalence of almost bipartite graphs}

\author[A Higashitani \and N. Matsumoto]{Akihiro Higashitani \and Naoki Matsumoto} 
\address[A Higashitani]{Department of Pure and Applied Mathematics, Graduate School of Information Science and Technology, Osaka University, Suita, Osaka 565-0871, Japan}
\email{higashitani@ist.osaka-u.ac.jp}
\address[N. Matsumoto]{Research Institute for Digital Media and Content, Keio University, Yokohama, Kanagawa 232-0062, Japan}
\email{naoki.matsumo10@gmail.com}
\thanks{}

\thanks{The first named author is supported by JSPS Grant-in-Aid for Scientific Research (C) 20K03513. 
The second named author is supported by JSPS Grant-in-Aid for Early-Career Scientists 19K14583.}

\begin{document}

\maketitle

\begin{abstract}
Two vertex colorings of a graph are Kempe equivalent if they can be transformed into each other 
by a sequence of switchings of two colors of vertices.
It is PSPACE-complete to determine 
whether two given vertex $k$-colorings of a graph are Kempe equivalent for any fixed $k\geq 3$,
and it is easy to see that every two vertex colorings of any bipartite graph are Kempe equivalent.
In this paper, we consider Kempe equivalence of {\it almost} bipartite graphs
which can be obtained from a bipartite graph by adding several edges 
to connect two vertices in the same partite set.
We give a conjecture of Kempe equivalence of such graphs,
and we prove several partial solutions and best possibility of the conjecture,
but it is more lately proved by Cranston and Feghali that this conjecture is false in general.
\end{abstract}

\section{Introduction}\label{sec:intro}

A {\it Kempe change} has been introduced by Alfred Kempe in his false proof of the four color theorem,
which is to switch two colors of 
a subgraph of a vertex-colored graph induced by vertices with only two colors.
Even today, this remains one of the fundamental and most powerful tools in graph coloring theory.
Two colorings are {\it Kempe equivalent}
if they are transformed into each other by a sequence of Kempe changes.
The Kempe equivalence of vertex colorings of graphs is deeply and widely studied,
since it is not only an important research subject in combinatorial reconfiguration
but also is related to many other subjects in graph theory;
for example, random coloring~\cite{vigoda2000improved},
perfectly contractile graph~\cite{bertschi1990perfectly},
a recoloring version of Hadwiger's conjecture~\cite{las1981kempe,bonamy2021recolouring} and so on.
Moreover, Kempe equivalence has many applications to various research fields;
e.g., statistical physics~\cite{mohar2009new}.
On the other hand, it is PSPACE-complete~\cite{bonamy2020diameter} to determine 
whether two given $k$-colorings of a graph $G$ are Kempe equivalent for any fixed $k\geq 3$,
even if $k=3$ and $G$ is a planar graph with maximum degree~6,
where a {\it $k$-coloring} is a proper vertex coloring with $k$ colors.
So it is hard in general to transform a given $k$-coloring into another $k$-coloring by Kempe changes.

Mohar~\cite{mohar2006kempe} surveyed the study of Kempe changes
and showed several fundamental results 
on Kempe equivalence of vertex/edge colorings.
He also proposed many interesting problems 
one of which concerns Kempe equivalence of regular graphs
and is completely solved in~\cite{bonamy2019conjecture}.
Here we introduce two results described in~\cite{mohar2006kempe} below, which are used to prove our results
(the second is originally proved in \cite{las1981kempe}).

\begin{prop}[\cite{mohar2006kempe}]\label{prop:bipar}
Let $G$ be a bipartite graph.
Every two $k$-colorings of $G$ with $k \geq 2$ are Kempe equivalent.
\end{prop}

A graph $G$ is {\it $d$-degenerate} if every subgraph of $G$ contains a vertex of degree at most $d$.

\begin{prop}[\cite{las1981kempe,mohar2006kempe}]\label{prop:dege}
Let $G$ be a $d$-degenerate graph with $d \geq 1$.
For any integer $k > d$, every two $k$-colorings of $G$ are Kempe equivalent.
\end{prop}

By these results,
it is an important problem 
to consider Kempe equivalence of two colorings of non-bipartite graphs without small maximum degree.
Then we focus on {\it almost} bipartite graphs
which can be obtained from a bipartite graph by adding several edges 
to connect two vertices in the same partite set.
Such a graph is one of most reasonable non-bipartite graphs,
and it is often investigated in various other contexts; 
for example, see~\cite{damaschke2003linear,kostochka2015minimum}.
In this paper,
we give an interesting conjecture on Kempe equivalence of almost bipartite graphs
and show partial solutions of it.
We also verify that the conjecture is best possible in some sense.
Moreover, we mention Kempe equivalence on color-critical graphs.

In the remaining of this section,
we precisely define terminologies used in this paper,
and then introduce our main conjecture and results.

\subsection{Definitions}\label{sec:def}

A graph $G$ is {\it $k$-colorable} 
if it has a map $c : V(G) \to \{1,2,\dots,k\}$ 
with $c(u) \neq c(v)$ for any $uv \in E(G)$,
where $V(G)$ and $E(G)$ denote the vertex and edge set of $G$, respectively,
and such a map is called a {\it $k$-coloring}.
The chromatic number of $G$, denoted by $\chi(G)$, 
is the minimum number $k$ such that $G$ is $k$-colorable.
A graph $G$ is {\it $k$-chromatic} if $\chi(G) = k$.
A graph $G$ is {\it $k$-critical}\/ if 
it is $k$-chromatic but every proper subgraph of $G$ is $(k-1)$-colorable.
An {\it odd wheel} is the graph obtained from an odd cycle $C$
by adding a vertex $v$ and joining $v$ to every vertex of $C$.
Note that every odd wheel is 4-chromatic and
the smallest odd wheel is $K_4$,
where $K_n$ denotes the complete graph with $n$ vertices.

For a vertex colored graph $G$,
$G(i,j)$ denotes the subgraph induced by all vertices colored with $i$ or $j$.
Every connected component $D$ of $G(i,j)$ is called a {\it Kempe component} (or {\it {\rm K}-component}),
and such a component $D$ is also called an {\it $(i,j)$-component}.
An edge $e \in E(G(i,j))$ is called an {\it $(i,j)$-edge}.
By switching the colors $i$ and $j$ on $D$, a new coloring can be obtained.
This operation is called a {\it Kempe change} (or {\it {\rm K}-change}).
In particular, we call a K-change on a component of $G(i,j)$ an {\it $(i,j)$-change}.
If a K-change is applied on a component containing a vertex $v$,
then such a K-change is also called a K-change {\it concerning} $v$.
Two $k$-colorings $c_1$ and $c_2$ of a graph $G$ are {\it Kempe equivalent} 
(or {\it {\rm K}-equivalent}),
denoted by $c_1 \sim_k c_2$,
if $c_1$ can be obtained from $c_2$ by a sequence of Kempe changes,
possibly involving more than one pair of colors in successive Kempe changes.
Let $\C_k = \C_k(G)$ be the set of all $k$-colorings of $G$.
The equivalence classes $\C_k / \sim_k$ are called the {\it {\rm K}$^k$-classes}.
The number of K$^k$-classes of $G$ is denoted by $\Kc(G,k)$.

For an integer $\ell \geq 0$,
a graph is called a {\it $B + E_{\ell}$ graph} 
(resp., a {\it $B + M_{\ell}$ graph})
if it is obtained from a bipartite graph $B$ by adding some $\ell$ edges
(resp., a matching of size $\ell$),
where a {\it matching} is a set of edges sharing no vertices each other.
For such graphs, 
the bipartite graph $B$ to which several edges are added is called a {\it based} bipartite graph,
and $E_{\ell}$ denotes the set of $\ell$ edges added.

For fundamental terminologies and notations undefined in this paper, 
we refer the reader to~\cite{bondy1976graph}.

\subsection{Results \& Problems}

For a $k$-colorable graph $G$,
if $|E(G)| < \binom{k}{2}$, then $G(i,j)$ is not connected for some distinct colors $i,j$.
In addition to this,
if $G$ is a $B + E_{\ell}$ graph,
then there is no edge in $E_{\ell} \cap E(G(i,j))$ for some distinct colors $i,j$.
This fact implies the following.

\begin{obs}\label{obs:1stset}
Let $G$ be a $k$-colorable $B+E_{\ell}$ graph.
If neither $S$ nor $T$ has an $(i,j)$-edge,
then the $k$-coloring of $G$ is Kempe equivalent to a $k$-coloring 
such that $i \notin C(S)$ and $j \notin C(T)$.
\end{obs}

Moreover, when $j \notin C(T)$,
if a connected component $H$ in the subgraph induced by $S$ has no vertex of color~$j$,
then we can make $H$ have a vertex of color~$j$ by applying an $(r,j)$-change to a vertex in $H$ with color $r$;
note that $H$ is an isolated vertex in $G(r,j)$.
Thus, we have the following observation.


\begin{obs}\label{obs:2ndset}
Let $G$ be a $k$-colorable $B+E_{\ell}$ graph with $\ell < \binom{k}{2}$,
and $S$ and $T$ be partite sets of the based bipartite graph.
Suppose that $G$ is already colored by $k$ colors so that $S$ (resp., $T$) has no vertex with color~$i$ (resp.,~$j$).
Then the $k$-coloring of $G$ can be 
transformed into a $k$-coloring by K-changes
such that each component of the subgraph induced by $S$ (resp., $T$)
has a vertex color~$j$ (resp.,~$i$).
\end{obs}

By these observations, 
if the number of added edges is less than $\binom{k}{2}$ for a $k$-colorable $B+E_{\ell}$ graph, 
then we can have a nice $k$-coloring by K-changes.
Such a $k$-coloring may help us to show the K-equivalence of every two colorings of a given graph.
Therefore, we conjecture the following, 
but this is more lately disproved for $k \geq 8$ 
by Cranston and Feghali~\cite{cranston2023kempe}.

\begin{conj}[False for $k \geq 8$]\label{conj:main}
Let $G$ be a $(k-1)$-colorable $B + E_{\ell}$ graph with $k \geq 4$ and $\ell < \binom{k}{2}$.
Then $\Kc(G,k) = 1$.
\end{conj}

Throughout this paper,
we show several partial solutions (for small $k$) 
and the sharpness of this conjecture.
(We do not consider the case when $k=3$ in the conjecture since $G$ is bipartite in that case.)

\smallskip
\noindent
\underline{$3$- or $4$-colorable case}

For Kempe equivalence of colorings of 3- or 4-colorable $B+E_{\ell}$ or $B + M_{\ell}$ graphs,
we have the following results.
Note that $B + M_{\ell}$ is 4-colorable for any $\ell \geq 0$
and that Theorem~\ref{thm:c3e5} is a positive partial solution of Conjecture~\ref{conj:main}.

\begin{thm}\label{thm:bm5}
Let $G$ be a $B + M_{\ell}$ graph with $k\geq 4$ and $\ell < \binom{k}{2}$.
Then $\Kc(G,k) = 1$.
\end{thm}

\begin{thm}\label{thm:c3e5}
Let $G$ be a $3$-colorable $B + E_{\ell}$ graph with $\ell \leq 5$.
Then $\Kc(G,4) = 1$.
\end{thm}

Theorem~\ref{thm:bm5} (resp., Theorem~\ref{thm:c3e5}) is best possible by the following result
(resp.,~by the $k=4$ case of Proposition~\ref{prop:general}).



\begin{prop}\label{prop:bm6exist3}
For any $k \geq 3$,
there are infinitely many $B + M_{\ell}$ graphs $G$ with $\ell = \binom{k}{2}$ and $\Kc(G,k) \geq 2$.
\end{prop}


\smallskip
\noindent
\underline{General case}

If the answer of Conjecture~\ref{conj:main} is yes,
then the statement is best possible by the following proposition.

\begin{prop}\label{prop:general}
For any integer $k \geq 4$, the following hold.
\begin{enumerate}
\item[(i)] There are infinitely many 
$(k-1)$-colorable $B + E_{\ell}$ graphs $G$ with $\ell = \binom{k}{2}$ and $\Kc(G,k) \geq 2$. 
\item[(ii)] There are infinitely many
$k$-chromatic $B + E_{\ell}$ graphs $G$ with $\ell = \binom{k}{2} - 1$ and $\Kc(G,k) \geq 2$.
\end{enumerate}
\end{prop}


We also give a partial positive solution of Conjecture~\ref{conj:main}, as follows.

\begin{thm}\label{thm:main}
Let $G$ be a $(k-1)$-colorable $B + E_{\ell}$ graph with $\ell < \binom{k}{2}$ and $k \geq 4$.
If every component $H$ induced by $\ell$ edges added is a path, 
a cycle of length at least~$4$ or a complete bipartite graph,
then $\Kc(G,k) = 1$.
\end{thm}

\begin{rem}
In fact,
we could prove Theorem~\ref{thm:main} 
even if the condition ``a cycle of length at least~$4$" is replaced with ``a cycle".
However, 
we prove the theorem with a cycle of length at least~$4$ in this paper
because the proof concerning a triangle is tedious (see Remark~\ref{rem:5-3}).
\end{rem}

\smallskip
\noindent
\underline{A special case}

A graph $G$ is {\it $k$-critical} (or {\it edge $k$-critical}\/) 
if $\chi(G) = k$ and every proper subgraph of $G$ is $(k-1)$-colorable.
It is easy to see that for $k=1,2$ and $3$, 
$k$-critical graphs are $K_1$, $K_2$ and odd cycles, respectively.
So, it is worth to consider $k$-critical graphs for $k \geq 4$.

Chen et al.~\cite{chen1997class} gave a construction 
producing a $B+M_{\ell}$ graph from an arbitrarily given graph.
The construction consists of two steps (see Figure~\ref{fig:g_const}):
\begin{figure}[htb]
\centering 
\input{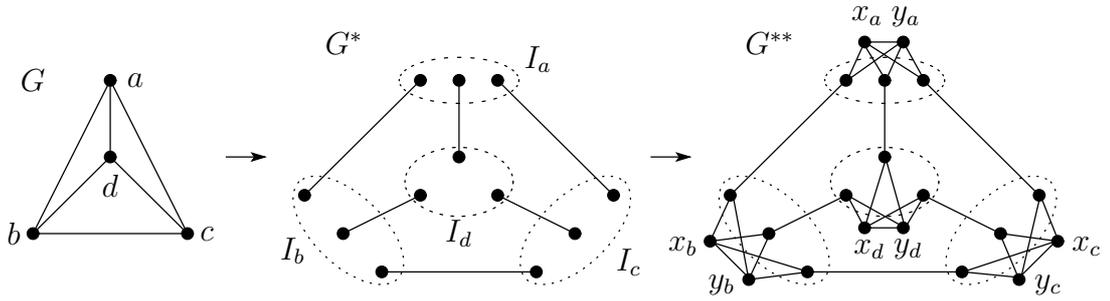}
\caption{Construction of $G^{**}$}
\label{fig:g_const}
\end{figure}

First, for each vertex $v$ of $G$, we construct an independent set $I_v$ of size $\deg_G(v)$
so that $I_v \cap I_u = \emptyset$ if $u \neq v$.
For each edge $uv$ of $G$, place an edge between $I_u$ and $I_v$ so that these edges are vertex disjoint.
Let $G^*$ be the resulting graph.
Note that $G^*$ consists of $|E(G)|$ disjoint edges.
Second, we add two new vertices $x_v$ and $y_v$ 
to each set $I_v$ joining them to each other and to all vertices of $I_v$.
We denote by $G^{**}$ the resulting graph. 
Note that $G^{**}$ is a connected 4-colorable $B + M_{\ell}$ graph 
and that it is 4-critical if $G$ is 4-critical~\cite{chen1997class}.
Moreover, since $G^{**}$ is 3-degenerate, 
the following holds by Proposition~\ref{prop:dege}.

\begin{cor}\label{thm:4cri}
For every connected graph $G$, $\Kc(G^{**},4)=1$.
\end{cor}

This corollary is a partial solution for our conjecture described later (Conjecture~\ref{conj:4cri}).
Moreover, Corollary~\ref{thm:4cri} implies that if a $B + M_{\ell}$ graph is isomorphic to $G^{**}$ for some given graph $G$,
then we can arbitrarily increase $\ell$ in Theorem~\ref{thm:bm5},
though $\ell < 6$ in general by Theorem~\ref{thm:bm5}.

Lately, Feghali~\cite{feghali2021kempe} proved that for every 4-critical planar graph $G$, $\Kc(G,4) = 1$
(which is conjectured in~\cite{mohar2006kempe}).
Moreover, it is not known whether a 4-critical graph $G$ with $\Kc(G,4) \geq 2$ exists
though there are infinitely many 3-colorable graphs $H$ with $\Kc(H,4) \geq 2$ 
(by Propositions~\ref{prop:bm6exist3} and~\ref{prop:general}),
and from our various considerations,
we are strongly convinced that for every 4-critical $B+M_{\ell}$ graph $G$, $\Kc(G,4) = 1$.
Thus, we propose the following conjecture.

\begin{conj}\label{conj:4cri}
For every $4$-critical graph $G$, $\Kc(G,4) = 1$.
\end{conj}

More generally, we also propose the following problem.

\begin{prob}\label{prob:critical}
Is it true that
for every $k$-critical graph $G$, $\Kc(G,k) = 1$?
\end{prob}

\subsection{Organization of the paper}

In the next section, we introduce several notations and lemmas,
and then we prove Propositions~\ref{prop:bm6exist3} and~\ref{prop:general} in Section~\ref{const}.
In Sections~\ref{sec:3},~\ref{sec:4} and~\ref{sec:5},
we prove Theorems~$\ref{thm:bm5}$,~$\ref{thm:c3e5}$ and~\ref{thm:main}, respectively.
In the final section, Section~\ref{sec:7}, we give a conclusion.

\section{Notations and Lemmas}\label{sec:2} 

\subsection{Notations}

For a graph $G$ and a subset $S \subseteq V(G)$, $G[S]$ denotes the subgraph induced by $S$.
For a vertex $v \in V(G)$, $N_G(v)$ denotes the set of neighbors of $v$.
For a $B+E_{\ell}$ graph $G$,
let $S$ and $T$ denote partite sets of a based bipartite graph $B$.
For a vertex subset $R$,
$C(R)$ denotes the set of colors used to vertices in $R$.
For the sake of simplicity,
let $E$ be the set of added $\ell$ edges
and let $E_S$ (resp., $E_T$) be the subset of $E$ 
in which each edge in the set joining two vertices in $S$ (resp., $T$).
If $E$ is a matching, i.e., $G$ is a $B+M_{\ell}$ graph,
then we denote $E,E_S$ and $E_T$ by $M,M_S$ and $M_T$, respectively.

In several figures referenced in the proofs,
we omit isolated vertices in $G[S]$ and $G[T]$ from those figures,
where $S$ and $T$ denote partite sets of a based bipartite graph $B$ of a $B+E_{\ell}$ (or $B+M_{\ell}$) graph $G$.
Moreover, bold lines in those figures denote edges in the edges added to a based bipartite graph.

\subsection{Lemmas}

We first prepare several lemmas to show our main results.
Throughout this subsection,
$G$ denotes a $k$-colorable $B+E_{\ell}$ graph and is already colored by $k$ colors.

\begin{lem}\label{lem:no_edge}
The following hold:
\begin{itemize}
\item[(i)]
If $i \notin C(S)$ and $E_T$ has no $(i,j)$-edge,
then the $k$-coloring of $G$ is Kempe equivalent to a $k$-coloring such that $j \notin C(T)$.
\item[(ii)]
If a vertex $v$ in $T$ with color $j$ is not adjacent to any vertex with $i$,
then the $k$-coloring $f$ of $G$ is Kempe equivalent to a $k$-coloring $g$ 
such that $j = f(v) \neq g(v) = i$ and $f(u) = g(u)$ for any other vertex $u \neq v$.
\end{itemize}
\end{lem}
\begin{proof}
(i) This immediately follows from Observation~\ref{obs:1stset}.\\
(ii) We clearly have the desired coloring by a $(i,j)$-change concerning $v$.
\end{proof}

\begin{lem}\label{lem:no_intersect}
Suppose $C(S) \cap C(T) = \emptyset$.
If each component in $G[E]$ is either a bipartite graph or an odd cycle,
then every two colorings of $S$ (or $T$) are K-equivalent.
\end{lem}
\begin{proof}
By the assumptions, the lemma trivially holds by Propositions~\ref{prop:bipar} and~\ref{prop:dege},
that is, every two colorings of $G[S]$ using only colors in $C(S)$ are Kempe equivalent 
by using only colors in $C(S)$.
\end{proof}

\begin{lem}\label{lem:no_change}
Let $c$ and $c'$ be two distinct $k$-colorings of $G$,
and suppose that every K-component in $c$ is connected
and that there is a K-component $H$ in $c'$ which is not isomorphic to any K-component in $c$.
Then $c$ and $c'$ are not K-equivalent.
\end{lem}
\begin{proof}
It is clear that 
any sequence of K-changes cannot transform a K-component in $c$ into $H$,
since every K-component in $c$ is connected.
\end{proof}

Finally, we show the Kempe equivalence of 3-colorings of a graph with a few edges.

\begin{lem}\label{lem:5edges}
Let $H$ be a $3$-colorable graph with $|E(H)| \leq 5$.
Then every two $3$-colorings of $H$ are Kempe equivalent.
\end{lem}
\begin{proof}
We may assume that $H$ is connected and 3-chromatic by Proposition~\ref{prop:bipar}. 
Then we have either (1) $H \cong C_5$ or (2) $H$ has a triangle.
In the case~(1), the lemma holds by Proposition~\ref{prop:dege}.
So we consider the case~(2).
If $H$ has two triangles sharing an edge,
then $H$ is {\it uniquely $3$-colorable}, that is, 
there is exactly one 3-coloring of $H$ up to permutation of the colors.
This trivially implies every two 3-colorings of $H$ are Kempe equivalent,
and hence, we may suppose that $H$ has exactly one triangle $xyz$.

Since $|E(H)| \leq 5$, there are exactly two vertices $u,v$ other than $x,y,z$.
There are three sub-cases by symmetry; $ux,vx \in E(G)$, $uv,ux \in E(G)$ or $ux,vy \in E(G)$.
Let $f$ and $g$ be distinct 3-colorings of $G$.
Note that every 3-coloring of $xyz$ can be changed to the same one as any other 3-coloring
by applying K-changes suitably.
Thus, we have $f(x)=g(x), f(y)=g(y)$ and $f(z)=g(z)$.
Moreover, in each sub-case,
if $f(u)\neq g(u)$ or $f(v) \neq g(v)$,
then we can easily see that 
the color $f(u)$ (resp., $f(v)$) of $u$ can be changed to $g(u)$ (resp., $g(v)$) 
by an $(f(u),g(u))$-change (resp., $(f(v),g(v))$-change) preserving the coloring of the triangle.
This completes the proof of the lemma.
\end{proof}

\section{Constructions}\label{const}

We first prove Proposition~$\ref{prop:bm6exist3}$.

\begin{proof}[Proof of Proposition~$\ref{prop:bm6exist3}$]
For any $k \geq 3$,
let $G$ be the graph shown in the left of Figure~\ref{fig:c3m6}.
The graph can be obtained from 
a complete bipartite graph $B$ with a matching of size~$\binom{k}{2}$ 
whose edges join two vertices in $S$ by removing edges of $B$ which join two vertices with the same color;
each pair of two colors in $\{1,2,\dots,k\}$ appears on some edge in $G[S]$
and each color in $\{1,2,\dots,k\}$ appears exactly once in $T$.
Note that $|S|=k(k-1)$ and $|T|=k$
and that $G$ has another $k$-coloring shown in the right of Figure~\ref{fig:c3m6} (using only three colors).
Since $G(i,j)$ is connected for any $i,j$ in the left of Figure~\ref{fig:c3m6}
and there is a K-component in the right coloring 
not isomorphic to any K-component in the left one,
these two $k$-colorings are not K-equivalent by Lemma~\ref{lem:no_change}.
Moreover, it is easy to see that the order of $G$ can be arbitrarily large 
by adding isolated vertices in $S$ or $T$ suitably. 
Therefore, this completes the proof of the proposition.
\end{proof}

\begin{figure}[htb]
\centering
\input{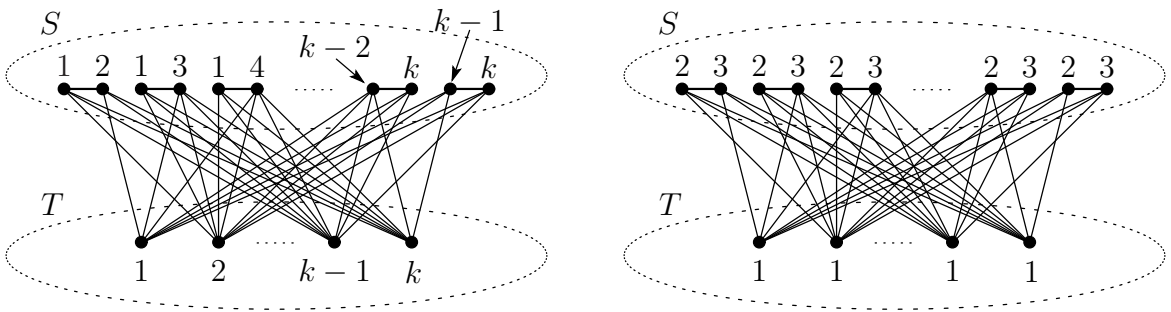}
\caption{A $B+M_{\ell}$ graph $G$ with $k \geq 3$, $\ell = \binom{k}{2}$ and $\Kc(G,k) \geq 2$}
\label{fig:c3m6}
\end{figure}

Next we prove Proposition~$\ref{prop:general}$.

\begin{proof}[Proof of Proposition~$\ref{prop:general}$]
We first construct the graph for $(i)$ shown in Figure~\ref{fig:g_k1}.
Prepare a complete bipartite graph $B$ with partite sets $S$ and $T$,
where $|S| = (k-2) + 2 \cdot (2k-3) = 5k - 8$ and $|T| = k$.
We add all of $\binom{k}{2}$ edges to $S$ so that 
those edges induce $K_{k-2} \cup (2k-3)K_2$,
and so $T$ consists of (exactly) $k$ isolated vertices. 
Then we give a $k$-coloring of the graph as shown in the left hand of Figure~\ref{fig:g_k1}:
We color vertices in $T$ by colors $1,2,\dots,k$,
$2k-3$ independent edges in $S$ by $(1,k-1), (2,k-1), \dots, (k-2,k-1), (1,k), (2,k), \dots, (k-1,k)$,
and $K_{k-2}$ by colors $1, 2, \dots, k-2$.
(This is possible since $k\geq 4$.)
According to the above coloring, 
we remove all edges between $S$ and $T$ joining two vertices with the same color.
The resulting graph is a $(k-1)$-colorable $B+E_{\ell}$ graph with $\ell = \binom{k}{2}$;
see the left hand of Figure~\ref{fig:g_k1}.
It is easy to check that every K-component of the coloring is connected,
but there is another $k$-coloring of the graph shown in the right of Figure~\ref{fig:g_k1},
and hence, 
those $k$-colorings are not K-equivalent by Lemma~\ref{lem:no_change}.
This completes the proof of $(i)$,
since the order of the above graph can be arbitrarily large 
as in the proof of Proposition~\ref{prop:bm6exist3}.

Next we construct the graph for $(ii)$ shown in Figure~\ref{fig:g_k2}.
Prepare a complete bipartite graph $B$ with partite sets $S$ and $T$,
where $|S| = k$ and $|T| = 3$.
We add an edge to $T$ joining two vertices of $T$,
$\binom{k-1}{2}$ edges to $S$ forming $K_{k-1}$,
and $k-3$ edges to $S$ joining the (unique) vertex $v$ not in the above $K_{k-1}$
and exactly $k-3$ vertices of the $K_{k-1}$.
(This is possible since $k\geq 4$.)
Similarly to the above case, 
we first give a $k$-coloring of the graph and then determine the based bipartite graph according to the given coloring
(see the left hand of Figure~\ref{fig:g_k2}):
Since $T$ contains exactly one edge $xy$ and an isolated vertex $z$,
we color $x$ (resp., $y,z$) by color $k-1$ (resp., $k$), respectively.
Then we color vertices of $K_{k-1}$ in $S$ by colors $1,2, \dots, k-1$
and $v$ by color $k-2$,
and hence, we set vertices in $S$ adjacent to $v$ be colored by colors $1,2, \dots, k-3$.
Finally, we remove two edges between $S$ and $T$; 
one of them joins $x$ and a vertex in $S$ with color $k-1$
and the other joins $y$ and a vertex with color $k-2$ which is not $v$.
The resulting graph is a $k$-chromatic $B+E_{\ell}$ graph with $\ell = \binom{k}{2} - 1$;
see the left of Figure~\ref{fig:g_k2}.
Every K-component of the coloring is connected, 
but there is another $k$-coloring of the graph shown in the right of Figure~\ref{fig:g_k2},
and hence those $k$-colorings are not K-equivalent by Lemma~\ref{lem:no_change}.
As in (i), the order of the constructed graph can be arbitrarily large,
and hence this completes the proof of $(ii)$.
\end{proof}

\begin{figure}[htb]
\centering
\input{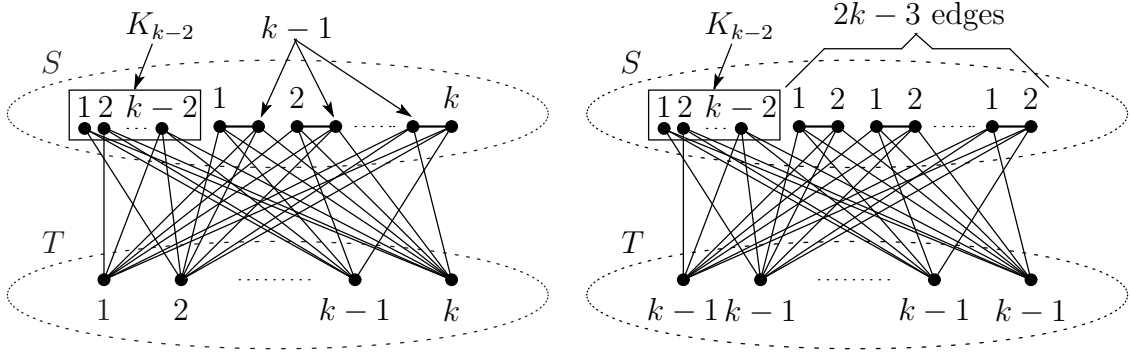}
\caption{Two non-K-equivalent $k$-colorings of a $(k-1)$-colorable $B+E_{\ell}$ graph
with $\ell = \binom{k}{2}$}
\label{fig:g_k1}
\end{figure}

\begin{figure}[htb]
\centering
\input{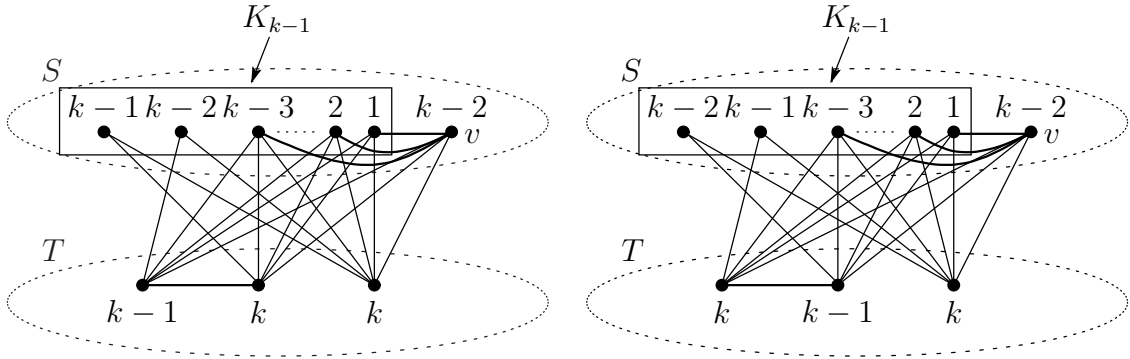}
\caption{Two non-K-equivalent $k$-colorings of a $k$-chromatic $B+E_{\ell}$ graph
with $\ell = \binom{k}{2} - 1$}
\label{fig:g_k2}
\end{figure}

\section{Proof of Theorem~$\ref{thm:bm5}$}\label{sec:3} 

Let $G$ be a $B+M$ graph with $k \geq 4$ and $|M| < \binom{k}{2}$.
Suppose that $G$ is colored by $k$ colors $1,2,\dots,k$.
In this proof, 
we show that every $k$-coloring of $G$ can be transformed into a 4-coloring of $G$ using colors in $\{1,2,3,4\}$
such that $C(S) \cap C(T) = \emptyset$.
This 4-coloring is called a {\it standard $4$-coloring} of $G$.
Note that every two standard 4-colorings are K-equivalent by Lemma~\ref{lem:no_intersect}.

We may assume that $G(1,k)$ contains no edge in $M$,
and so we have $C(S) \subseteq \{1,2,\dots,k-1\}$ 
and $C(T) \subseteq \{2,3,\dots,k\}$ by Observation~\ref{obs:1stset}.
By applying Lemma~\ref{lem:no_edge}(ii) repeatedly, 
we can make every edge in $G[T]$ have color~$k$.
We divide the remaining proof into two cases.

\medskip
\noindent
{\bf Case 1.} $M_T$ has no $(i,k)$-edge for some $i \in \{2,3,\dots,k-1\}$.

Without loss of generality, we set $i = 2$.
By Lemma~\ref{lem:no_edge}(i), we have $1,2 \notin C(T)$.  
Then we can obtain $C(S) \subseteq \{1,2\}$ by $(1,i)$- and  $(2,i)$-changes for $i \in \{3,\dots,k-1\}$.
Therefore, since we can also obtain $C(S) \subseteq \{3,4\}$ by Lemma~\ref{lem:no_edge}(ii) again,
the resulting coloring is a standard 4-coloring.

\medskip
\noindent
{\bf Case 2.} $M_T$ has an $(i,k)$-edge for every $i \in \{2,3,\dots,k-1\}$.

If $M_S$ contains a $(2,3)$-edge $uv$ where $u$ is colored with~2,
then the color of $u$ can be changed to color~$1$ by a $(1,2)$-change concerning $u$ since $1 \notin C(T)$.
Thus, we can make $M_S$ contain no $(2,3)$-edge by repeated application of such $(1,2)$-changes.
Moreover, 
since every edge in $G[T]$ has color $k$,
there is no $(2,3)$-edge in $M_T$, either.
Thus, by Observation~\ref{obs:1stset},
we can make $1,2 \notin C(T)$ by $(2,3)$-changes, 
and hence, this case can be deduced to Case~1.
\qed

\section{Proof of Theorem~$\ref{thm:c3e5}$}\label{sec:4} 

Let $G$ be a 3-colorable $B+E$ graph with $|E| \leq 5$.
Suppose that $G$ is colored by using all of four colors $\{1,2,3,4\}$.
Since $|E| \leq 5$,
there are two colors, say 3 and 4, such that $G(3,4)$ contains no edge in $E$.
By Observation~\ref{obs:1stset}, we may assume that $4 \notin C(S)$ and $3 \notin C(T)$,
and by Observation~\ref{obs:2ndset}, each isolated vertex in $S$ (resp., $T$) has color~3 (resp.,~4).
By this fact, if $E_S = \emptyset$ or $E_T = \emptyset$, say the latter, that is, $T$ is an independent set,
then we can easily see that every two 4-colorings of $G$ are K-equivalent by Lemma~\ref{lem:5edges}.
So we may assume by symmetry that $|E_S| \geq |E_T| \geq 1$,
and it suffices to consider the following two types;
(1) both $G[S]$ and $G[T]$ are bipartite and (2) $G[S]$ contains a triangle $xyz$, where $x,y,z \in V(G)$.
Then we show the following two claims,
where the desired 4-coloring in each of Claims~\ref{cl:thm_1} and~\ref{cl:thm_2} is called a standard 4-coloring.

The proof consists of the following three steps:
\begin{itemize}
\item[1.]
For the type (1), every 4-coloring can be transformed into a standard 4-coloring 
by Claim~\ref{cl:thm_1}.
\item[2.]
For the type (2), every 4-coloring can be transformed into a standard 4-coloring
by Claim~\ref{cl:thm_2}.
\item[3.]
For types (1) and (2), 
every two standard 4-colorings are K-equivalent 
by Lemma~\ref{lem:no_intersect} and Claim~\ref{cl:thm_3}, respectively.
\end{itemize}

\begin{cl}\label{cl:thm_1}
If both $G[S]$ and $G[T]$ are bipartite,
then every $4$-coloring of $G$ can be transformed into a $4$-coloring such that $C(S)=\{1,3\}$, $C(T)=\{2,4\}$ 
and all isolated vertices in $G[S]$ (resp., $G[T]$) are colored by color~$3$ (resp.,~$4$).
\end{cl}
\begin{proof}[Proof of Claim~$\ref{cl:thm_1}$]
We divide the proof into two cases.

\smallskip
\noindent
{\bf Case 1.} $E_T$ has no $(1,4)$-edge (the case when it has no $(2,4)$-edge is similar).

By Lemma~\ref{lem:no_edge}(i), we have $C(T) = \{2,4\}$ since $4 \notin C(S)$.
Let $v$ be a vertex in $S$ with color~2 (if it does not exist, then we are done).
Since $G[S]$ is bipartite, 
two vertices adjacent to $v$ with colors~1 and~3 are not in the same K-component of $G(1,3)$;
otherwise, $G[S]$ has an odd cycle, a contradiction.
Thus, by a sequence of $(1,3)$-changes, all neighbors of $v$ in $S$ have color~1.
Hence we can change the color of $v$ to~3 by a $(2,3)$-change concerning $v$ since $3 \notin C(T)$.
By repeated applications of the above operation to every vertex in $S$ with color~2,
we have $C(S)=\{1,3\}$ and $C(T)=\{2,4\}$, that is, we can finally obtain a standard 4-coloring of $G$.

\smallskip
\noindent
{\bf Case 2.} $E_T$ has both a $(1,4)$-edge and a $(2,4)$-edge.

By this assumption, $|E_S| \geq |E_T|=2$.
Let $uv,u'v'$ be edges in $E_T$, where $u,v,u'$ and $v'$ have colors~2,~4,~1 and~4, respectively (possibly, $v = v'$);
see Figure~\ref{fig:type1p}.
(We show that $G$ will be the graph shown in Figure~\ref{fig:type1p}.)
\begin{figure}[htb]
\centering
\input{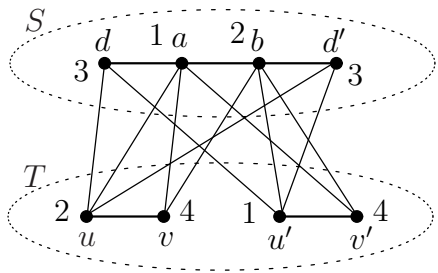}
\caption{The case when $E_T$ has both a $(1,4)$-edge and a $(2,4)$-edge}
\label{fig:type1p}
\end{figure}

In this case, our aim is to make $C(T)=\{2,4\}$ by K-changes;
if we attain this aim, then this case can be deduced to Case~1.
If $u'$ is not adjacent to a vertex in $S$ with color~2 or $S$ does not contain a $(1,2)$-edge,
then we are done by Lemma~\ref{lem:no_edge}.
Thus, $E_S$ has a $(1,2)$-edge $ab$ in which $a$ and $b$ have colors~1 and~2, and $ua,u'b \in E(G)$.

Observe that if the color of $a$ or $b$ can be replaced with color~3 without changing the color of other vertices,
then we are done by Lemma~\ref{lem:no_edge} as above.
Thus, we may assume by Lemma~\ref{lem:no_edge}(ii) that 
both $a$ and $b$ are adjacent to a vertex with color~3 in $G[S]$.
Let $d$ and $d'$ be vertices in $S$ with color~3 adjacent to $a$ and $b$, respectively.
Since $G[S]$ is bipartite, $d \neq d'$, and hence, $|E_S|=3$.
By similar argument, we have $ud,ud',u'd,u'd'\in E(G)$;
for example, if $ud \notin E(G)$, 
then we can make $a$ be adjacent to no vertex with color~3 by Lemma~\ref{lem:no_edge}(ii),
that is, this case can be deduced to the above.
Moreover, we also have $va,vb,v'a,v'b \in E(G)$ by Lemma~\ref{lem:no_edge}(ii),
as shown in Figure~\ref{fig:type1p};
for example, if $va\notin E(G)$,
then the color $v$ can be changed to color~1 by Lemma~\ref{lem:no_edge}(ii)
and hence the case can be deduced to Case~1 since $E_T$ has no $(2,4)$-edge.
Note that $au', bu, dv, dv', d'v, d'v' \notin E(G)$ since $G$ contains no odd wheel.

Observe that every isolated vertex in $S \cap (N_G(v) \cup N_G(v'))$ is not adjacent to $u$ or $u'$,
since otherwise $G$ contains an odd wheel.
Thus, we can make $C(T)=\{2,4\}$, as follows:
\begin{itemize}
\item[Step 1.] Replace the color of every isolated vertex adjacent to $v$ or $v'$ with color~1 or~2 by K-changes 
on $(1,3)$- or $(2,3)$-components consisting of only the isolated vertex.
\item[Step 2.] Change colors of $v,v'$ to color~3 by $(3,4)$-changes.
\item[Step 3.] Change colors of $u,u'$ to color~4 by $(2,4)$- and $(1,4)$-changes.
\item[Step 4.] Change colors of $v,v'$ to color~2 by $(2,3)$-changes,
and then restore the previously changed color of all isolated vertices in Step 1 (if necessary).
\end{itemize}
This completes the proof of the claim. 
\end{proof}

\begin{cl}\label{cl:thm_2}
If $G[S]$ contains a triangle $xyz$, where $x,y,z \in V(G)$,
then every $4$-coloring of $G$ can be transformed into a $4$-coloring such that
$C(S)=\{1,2,3\}$, $C(T) \subseteq \{1,2,4\}$, $(f(x),f(y),f(z)) = (1,2,3)$,
and all isolated vertices in $G[S]$ (resp., $G[T]$) are colored by color~$3$ (resp.,~$4$).
\end{cl}
\begin{proof}[Proof of Claim~$\ref{cl:thm_2}$]
Let $uv$ be an edge in $E_T$ and let $ww'$ be an edge in $E \setminus \{xy,yz,zx,uv\}$;
note that it suffices to show the case when $|E| = 5$.
Up to $(1,2)$-changes,
we may assume without loss of generality
that a 4-coloring $h$ of $G$ satisfies one of the following:
\begin{itemize}
\item[] (i) $(h(x),h(y),h(z)) = (1,3,2)$ \quad  or \quad (ii) $(h(x),h(y),h(z)) = (3,2,1)$.
\end{itemize}

In this proof, we consider the case (i) since the case (ii) is similar; the role of colors~1 and~2 is switched
(in the proof of (ii), we exchange the colors 1 and 3 of $x$ and $z$
instead of exchanging the colors 2 and 3 of $y$ and $z$ in the following proof).
If we can apply a $(2,3)$-change concerning $y$ and $z$ preserving colors of vertices in $T$,
then we are done,
where this K-change is denoted by {\it $R_{yz}$}.
So, we may suppose that $h(u)=2$ and $uy \in E(G)$.
If $u$ is adjacent to no vertex with color~4 (resp.,~1),
then we can change the color of $u$ to~4 (resp.,~1) 
by Lemma~\ref{lem:no_edge}(ii) and apply $R_{yz}$.
Hence, without loss of generality, we may suppose $h(v)=4$,
and we have one of the following (see Figure~\ref{fig:type2p}):
\begin{itemize}
\item[(1)] $ux \notin E(G)$, $ww' \in E_T$, $h(w) = 1$, $w' = u$ and $wz \in E(G)$.
\item[(2)] $ux \in E(G)$, $vz \notin E(G)$, $ww' \in E_S$, $vw \in E(G)$ (i.e., $w \neq z$), 
$h(w) = 2$ and $h(w') = 3$.
\item[(3)] $ux \in E(G)$ and $vz \in E(G)$.
\end{itemize}

\begin{figure}[htb]
\centering
\input{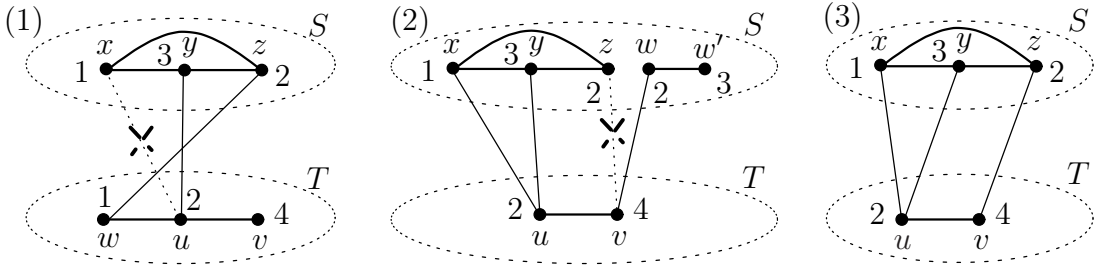}
\caption{Cases (1), (2) and (3) of the proof of Type 2}
\label{fig:type2p}
\end{figure}

Here we describe the reasons why the cases (1) and (2) occur
(note that the case (3) is naturally derived if neither (1) nor (2) occurs):
\begin{description}
\item[(1)]
Suppose $ux \notin E(G)$.
If $ww' \in E_S$,
then $h(w) = 1$, $uw \in E(G)$ and $w'=z$;
otherwise, we can change the color of $u$ to~1 without changing the colors of $x,z$.
However, in this case, 
we can change the color of $w$ to~3 by Lemma~\ref{lem:no_edge}(ii),
and hence, we can apply $R_{yz}$ after changing the color of $u$ to~1.
So $ww' \in E_T$, $h(w)=1$ and $w' = u$.

\item[(2)]
Suppose $ux \in E(G)$ but $vz \notin E(G)$.
If we can exchange colors of $u$ and $v$ by $(2,4)$-change 
without changing any color of vertices in $S$,
then we are done by $R_{yz}$ after that change.
So we have $vy \in E(G)$.
In this case, $xv \notin E(G)$ since $G$ has no odd wheel,
and hence, $ww' \in E_S$, $vw \in E(G)$, $h(w)=1$ and $h(w')=3$;
note that if $h(w') \neq 3$, then we can change the color of $w$ to~3 
by Lemma~\ref{lem:no_edge}(ii).
However, since $G$ has no odd wheel, $w$ or $w'$ is not adjacent to a vertex with color~2.
Therefore, after making $1\notin \{h(w),h(w')\}$ by several K-changes,
we can transform the color of $v$ and $u$ to~1 and~4
by $(1,4)$- and $(2,4)$-changes, respectively, and then apply $R_{yz}$.
By the above argument, 
we have $ww' \in E_S$, $vw \in E(G)$, $h(w)=2$ and $h(w')=3$;
the last condition is similarly derived as above.
\end{description}

Now we prove the three cases in order.
In the case (1), we can make $u$ adjacent to no vertex of color~1 
by a $(1,4)$-change concerning $w$ by Lemma~\ref{lem:no_edge}(ii),
and hence,
we can apply $R_{yz}$ after changing the color of $u$ to color~1.
We consider the case (2).
If we can apply a $(1,4)$-change concerning $v$ preserving colors of vertices in $S$,
then we are done by a $(2,4)$-change concerning $u$ and $R_{yz}$,
and hence, $vx \in E(G)$.
This implies $vy \notin E(G)$ since $G$ has no odd wheel.
Thus, 
after changing the color of $w$ to color~1 by Lemma~\ref{lem:no_edge}(ii),
we can exchange colors of $u$ and $v$ by a $(2,4)$-change and apply $R_{yz}$.

Thus, we consider the case (3).
By similar observations as above, we have two sub-cases (see Figure~\ref{fig:3-ab}):
\begin{itemize}
\item[(3-a)] $vx \notin E(G)$, $ww' \in E_S$, $h(w)=1$, $h(w')=3$ and $uw,uw' \in E(G)$.
\item[(3-b)] $vx \in E(G)$.
\end{itemize}

\begin{figure}[htb]
\centering
\input{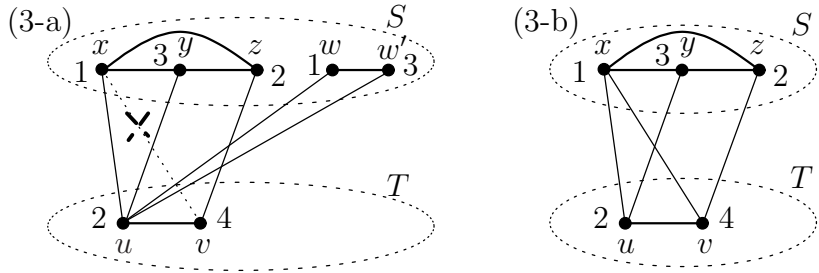}
\caption{Cases (3-a) and (3-b)}
\label{fig:3-ab}
\end{figure}

\smallskip
\noindent
\underline{Case (3-a).}
If we can change the color of $v$ to~1 without changing colors of vertices in $S$,
then this case can be deduced to the case (2).
So, we try to make a coloring such that we can apply the above K-change.
If $w' \notin \{x,y,z\}$, then $vw' \in E(G)$;
otherwise, we obtain a desired coloring 
by exchanging the colors of $w$ and $w'$ by a K-change.
Since $G$ has no odd wheel, $uw' \notin E(G)$.
This means that we obtain a desired coloring 
by applying $(1,3)$- and $(1,2)$-changes to $w$ and $w'$, respectively.
So $w' = y$ and $vy \notin E(G)$ since $G$ has no odd wheel.
Therefore, we have a standard 4-coloring of $G$ by applying K-changes in the following order:
\begin{itemize}
\item $(1,3)$-change concerning only $w,x,y$.
\item $(1,4)$-change concerning only $v$.
\item $(1,2)$-change concerning only $u,v,y,z$.
\item $(1,4)$-change concerning only $u$.
\item $(1,3)$-change concerning only $x,z$.
\end{itemize}

\smallskip
\noindent
\underline{Case (3-b).}
Now $vy \notin E(G)$ since $G$ has no odd wheel.
If every vertex in $S \setminus \{x,y,z\}$ with color~3 has no neighbor with color~1 in $G$,
then we apply the following series of K-changes, denoted by $Q_{1,3}$:
\begin{itemize}
\item $(1,3)$-changes concerning the colors of vertices with color~3 in $S \setminus \{x,y,z\}$.
\item $(3,4)$-change concerning only $v$.
\item $(2,3)$-change concerning only $u,v,y,z$.
\item $(3,4)$-changes concerning $u$
and $(1,3)$-changes concerning isolated vertices with color~1 in $G[S]$.
\end{itemize}

If $Q_{1,3}$ cannot be directly applied to $G$,
then we have one of the following (see Figure~\ref{fig:type2pcd}):
\begin{itemize}
\item[(c)] $ww'\in E_S$, $h(w)=3, h(w')=1$ and $vw \in E(G)$;
the final condition holds to avoid the second step of $Q_{1,3}$.
\item[(d)] $ww'\in E_T$, $h(w)=1$, $h(w')=4$ and $wz,w'z \in E(G)$;
for example, if $wz \notin E(G)$,
then after changing the color of $w$ to~2 by Lemma~\ref{lem:no_edge}(ii), 
we can apply $Q_{1,3}$.
\end{itemize}

\begin{figure}[htb]
\centering
\input{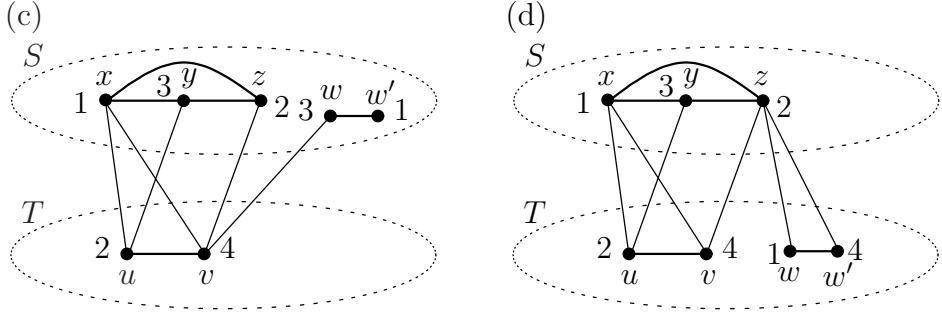}
\caption{Two sub-cases (c) and (d) of the case (3-b)}
\label{fig:type2pcd}
\end{figure}

In what follows, our aim is to make that
every vertex in $S \setminus \{x,y,z\}$ with color~3 has no neighbor with color~1 in $G$.
We first consider the case (c).
Suppose $w,w' \notin \{x,y,z\}$.
Note that both $w$ and $w'$ are adjacent to $u$;
for example, if $uw \notin E(G)$,
then the colors of $w$ and $w'$ can be changed to $2$ and $3$ 
by $(1,3)$- and $(1,2)$-changes concerning $w,w'$ and only $w$, respectively.
This implies $vw' \notin E(G)$ since $G$ has no odd wheel,
but in this case,
we can directly execute $Q_{1,3}$ 
since no vertex in $T$ has color~3 after the first operation in $Q_{1,3}$.
Thus, $w' = x$ since $vy \notin E(G)$, and moreover, $uw \notin E(G)$ since $G$ has no odd wheel.
Therefore,
we can change the color of $w$ to color~2 by a $(2,3)$-change by Lemma~\ref{lem:no_edge}(ii), 
and hence, we can execute $Q_{1,3}$.

Next we consider the case (d).
Similarly to the previous case,
we have a standard 4-coloring of $G$ by applying K-changes in the following order:
\begin{itemize}
\item We temporarily change colors of isolated vertices in $G[T]$ adjacent to $y$ to 1 or 2;
this is applicable since at least one of $x$ and $z$ is not adjacent to such a vertex.
\item $(1,4)$-change concerning $w,w'$ if $yw' \in E(G)$;
otherwise, $(3,4)$-change concerning only $y$.
(Note that if $yw' \in E(G)$ (i.e., $v \neq w'$), then $yw,xw' \notin E(G)$.)
\item $(2,4)$-change concerning $u,v,y,z,w'$ and several isolated vertices in $G[T]$.
\item $(3,4)$-change concerning only $z$.
\item $(1,4)$- or $(2,4)$-changes concerning isolated vertices in $G[T]$ with color~1 or~2
to change colors of all isolated vertices in $G[T]$ to color~4.
\end{itemize}
This completes the proof of the claim.
\end{proof}

We finally show the following.

\begin{cl}\label{cl:thm_3}
Every two standard 4-colorings of type $(2)$ are K-equivalent.
\end{cl}
\begin{proof}[Proof of Claim~$\ref{cl:thm_3}$]
Let $u,v,w,w',x,y,z$ be the same as in Claim~\ref{cl:thm_2}.
Let $f$ be a standard 4-coloring of $G$,
that is, $C(S)=\{1,2,3\}$, $C(T) \subseteq \{1,2,4\}$, 
$f(x)=1$, $f(y)=2$, $f(z)=3$
and $f(p)=3$ (resp., $f(p)=4$) for each isolated vertex $p$ in $G[S]$ (resp.,~$G[T]$).
We may suppose that there is a distinct standard 4-coloring $g$ other than $f$,
and we suffice to prove that $f$ can be changed to $g$ by K-changes.

\medskip
\noindent
{\bf Case 1.} $f(u)=g(u) = c_u$, $f(v)=g(v) = c_v$ but $(f(w),f(w')) \neq (g(w),g(w'))$.

Suppose $ww' \in E_S$ and $f(w) \neq g(w)$.
Note that $w$ is not adjacent to any vertex in $T$ with color $g(w)$ 
since $f(u)=g(u)$ and $f(v)=g(v)$.
So, if a $(f(w),g(w))$-component $D$ with $w$ contains a vertex in $T$,
then we have $f(w') = g(w) = c^g_{w}$, $f(u)=g(u)=f(w) = c_u$ and $uw' \in E(G)$; 
see Figure~\ref{fig:case1_fgs}.
(Note that if $w' \in \{x,y,z\}$, then $D$ consists only of $w$.)

\begin{figure}[htb]
\centering
\input{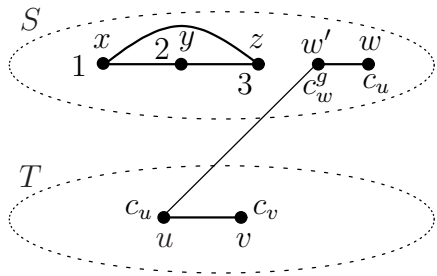}
\caption{The coloring $f$ in Case 1}
\label{fig:case1_fgs}
\end{figure}

In this case, $c^g_{w'} = g(w') \neq g(u) = f(w) = c_u$, and hence,
we can change the color of $w'$ to $c^g_{w'}$ 
by a $(c^g_{w},c^g_{w'})$-change by Lemma~\ref{lem:no_edge}(ii).
After that, we can change the color of $w$ by Lemma~\ref{lem:no_edge}(ii) again.
Since the similar argument as above follows even if $ww' \in E_T$,
we can obtain the coloring $g$.

\medskip
\noindent
{\bf Case 2.} $f(u)=g(u)$ but $f(v)\neq g(v)$.

Let $R_v$ be a K-change on a $(f(v),g(v))$-component $H$ with $v$,
and we show that $R_v$ can be applied preserving the colors of $x,y,z$ and $u$ after several K-changes.
If $H$ consists of only $v$, then we are done by Lemma~\ref{lem:no_edge}(ii).
Otherwise, $H$ contains $w$ (i.e., $f(w) = g(v)$) and $w \notin \{x,y,z\}$;
note that $g(v) \neq g(u) = f(u)$.
If $ww' \in E_T$, i.e., $w'=v$, regardless of colors of $g(w)$,
we can apply a $(f(w),g(w))$-change concerning $w$ by Lemma~\ref{lem:no_edge}(ii).
Hence, after that, we can apply $R_v$.
So suppose $ww' \in E_S$, and we have the following two sub-cases by symmetry 
(see Figure~\ref{fig:case2_subcases}):
\begin{itemize}
\item[1.] $f(v) = 4$, $g(v)=f(w)= 1$ and $w \neq x$.
\item[2.] $f(v) = 1$, $g(v) = f(w)=2$ and $w'=x$.
\end{itemize}

\begin{figure}[htb]
\centering
\input{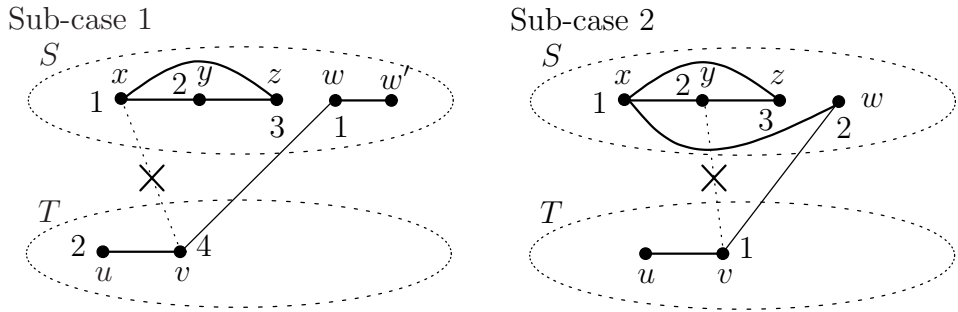}
\caption{Two sub-cases 1 and 2 of Case 2 where each of colorings denotes the 4-coloring $f$}
\label{fig:case2_subcases}
\end{figure}

\smallskip
\noindent
\underline{Sub-case 1.} 
We have $f(u) = g(u) = 2$ by the assumption.
If $w$ is not adjacent to any vertex of color~3,
then we can apply $R_v$ after a $(1,3)$-change concerning $w$, since $3 \notin C(T)$.
For the case when $w$ is not adjacent to any vertex of color~2, we are done similarly to the above.
Thus, $f(w') = 3$ and $uw\in E(G)$.
Now, if $w' = z$, 
then 
$g(w) \notin \{1,2,3\}$ since $g(u)=2$ and $g(v)=1$,
which contradicts the 3-colorability of $G$.
So we have $w' \neq z$, $vw' \in E(G)$ 
(otherwise, we can apply $R_v$ after a $(1,3)$-change concerning $w$ and $w'$),
and $uw' \notin E(G)$ since $G$ has no odd wheel.
In this case, 
we can change the color of $w'$ to color~2 by a $(2,3)$-change by Lemma~\ref{lem:no_edge}(ii),
and hence we can apply $R_v$ after a $(1,3)$-change concerning $w$.

\smallskip
\noindent
\underline{Sub-case 2.} 
This sub-case can be proved similarly to Sub-case~1, that is,
we can apply $R_v$ after a $(2,3)$-change concerning $w$ by Lemma~\ref{lem:no_edge}(ii), 
since $f(u) = g(u) \neq g(v)$ and $3 \notin C(T)$.

\medskip
\noindent
{\bf Case 3.} $f(u) \neq g(u)$.

We may assume that $f(v)\neq g(v)$ by Case~2.
We consider the following sub-cases:
\begin{itemize}
\item[(a)] $g(u)=4$ (the case when $g(v)=4$ is similar).
\item[(b)] $g(u)=2$, $g(v)=1$ and $f(u) \in \{1,4\}$.
\end{itemize}

\smallskip
\noindent
\underline{Sub-case (a).}  
We may suppose that $f(u)=1$, i.e., $ux \notin E(G)$. 
If we can apply a $(1,4)$-change concerning $u$ preserving colors of vertices $x,y$ and $z$, 
then we are done.
Otherwise, as shown in Figure~\ref{fig:case3-i-ii}, we have either 
\begin{itemize}
\item[(a)-(i)] $f(v)=2$, $f(w)=4$, $w' = u$ (i.e., $ww' \in E_T$) and $wx\in E(G)$, or
\item[(a)-(ii)] $f(v) = 4$. 
\end{itemize}

\begin{figure}[htb]
\centering
\input{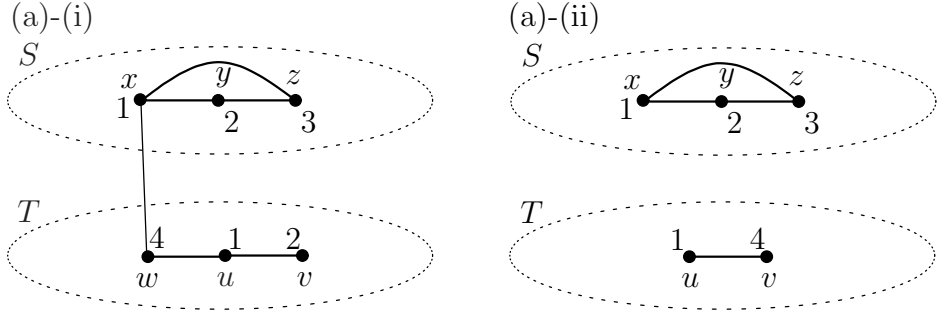}
\caption{Sub-cases (a)-(i) and (a)-(ii)}
\label{fig:case3-i-ii}
\end{figure}

In (a)-(i), note that $g(w)=2$ and $g(v)=1$ (otherwise, the case can be deduced to Case~2).
Thus, we change the color of $w$ to~2 by a $(2,4)$-change by Lemma~\ref{lem:no_edge}(ii),
and then the case is deduced to Case~2 by labeling $w,u,v$ as $u,v,w$, respectively.

So we consider (a)-(ii).
Note that $g(v)=2$, i.e., $vy \notin E(G)$.
So, similarly to the proof of Case~2,
we may suppose that $ww' \in E_S$ and 
that $f(w)=2$, $vw \in E(G)$ (i.e., $w \neq y$) and $f(w')=3$.
If $w' \neq z$,
then we also have $vw'\in E(G)$ but at least one of $uw$ and $uw'$ does not exist 
since $G$ has no odd wheel.
This implies that we can change the color of $v$ to $2$ by applying K-changes in the following order;
\begin{itemize}
\item $(2,3)$-change concerning $w,w'$ (if $uw' \notin E(G)$).
\item $(1,2)$-change concerning only $w$ or $w'$. 
\item $(2,4)$-change concerning only $v$. 
\end{itemize}
Thus, we have $w' = z$.
By the assumptions, we have $uw,uy,vx \in E(G)$ by the above arguments 
(see Figure~\ref{fig:case_a_ii});
for example, if $uy \notin E(G)$, 
then we are done by applying $(1,2)$- and $(2,4)$-changes concerning $u,w$ and $u,v$, respectively.

\begin{figure}[htb]
\centering
\input{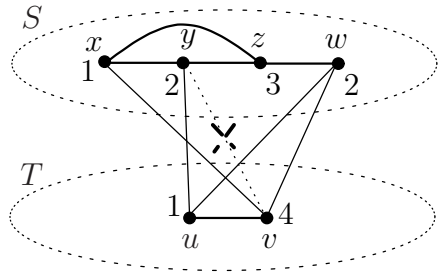}
\caption{The configuration of the case (a)-(ii)}
\label{fig:case_a_ii}
\end{figure}

In this case, $uz \notin E(G)$ or $vz \notin E(G)$ since $G$ has no odd wheel.
Therefore, we can change the color of $u$ to $4$ by applying K-changes as follows:
If $uz \notin E(G)$, then we apply K-changes in the following order.
\begin{itemize}
\item $(1,3)$-change concerning $u$ (and several isolated vertices in $G[S]$ adjacent to $u$).
\item $(1,2)$-change concerning only $w$. 
\item $(2,4)$-change concerning only $v$. 
\item $(3,4)$-change concerning only $u$. 
\item $(1,3)$-changes concerning isolated vertices in $G[S]$ with color~1.
\end{itemize}
If $vz \notin E(G)$, then we apply K-changes in the following order.
\begin{itemize}
\item $(2,3)$-changes concerning isolated vertices in $G[S]$ adjacent to $v$ to change the color of such a vertex to~2.
\item $(3,4)$-change concerning only $v$. 
\item $(1,4)$-change concerning only $u$. 
\item $(1,2)$-change concerning only $w$. 
\item $(2,3)$-change containing $v$ and isolated vertices in $G[S]$ adjacent to $v$.
\end{itemize}

\smallskip
\noindent
\underline{Sub-case (b).} 
If $f(u)=4$, then $f(w)=2$, $uw \in E(G)$ (i.e., $ww' \in E_S$) and $w \neq y$;
note that if $ww' \in E_T$, then the case can be immediately deduced to Case~2 as in (a)-(i).
In this case, $f(v) = 2$ and $v$ is adjacent to neither $x$ nor $y$ in $G$,
and hence, we can change the color of $v$ to $1=g(v)$ by a $(1,2)$-change preserving colors of $x,y$ and $z$.
Hence, we suppose $f(u)=1$.
If $f(v)=4$,
then by the similar argument as (a)-(ii),
we have $ww' \in E_S, f(w)=2, f(w')=1$ and $uw, vw' \in E(G)$.
In this case, we see $w,w' \notin \{x,y,z\}$, and hence, 
we can change the color of $u$ to $2=g(u)$ by a $(1,2)$-change preserving colors of $x,y$ and $z$.
So suppose $f(v)=2$.
If a $(1,2)$-change concerning $u$ changes the colors of $x,y$, 
then we have either
$ww' \in E_T$, $f(w)=1$, $wy \in E(G)$ and $w'=v$, or 
$ww' \in E_S$, $f(w)=1$, $vw \in E(G)$ and $w'=y$, 
by symmetry.
However, we can change the color of $v$ to color~4 by a $(2,4)$-change in either case,
and hence, we are done similarly to the previous argument for $f(v)=4$.
\end{proof}

Therefore, by Claims~\ref{cl:thm_1},~\ref{cl:thm_2} and~\ref{cl:thm_3},
every two 4-colorings are K-equivalent in each of two types (1) and (2), 
which completes the proof of the theorem.
\qed

\section{Proof of Theorem~\ref{thm:main}}\label{sec:5}

Before we start proving Theorem~$\ref{thm:main}$,
we show the following proposition.

\begin{prop}\label{prop:5colors}
Let $G$ be a $4$-colorable $B+E_{\ell}$ graph with $\ell < \binom{5}{2}=10$.
If every component $H$ induced by $\ell$ edges added is a path, 
a cycle of length at least~$4$
or a complete bipartite graph,
then $\Kc(G,5)=1$.
\end{prop}
\begin{proof}
Let $G$ be a $4$-colorable $B+E_{\ell}$ graph with $\ell = |E| < \binom{5}{2}=10$.
Suppose that $G$ is colored by five colors $1,2,\dots,5$, and let $f$ be a $5$-coloring of $G$.

Since $|E| < \binom{5}{2}=10$,
there are two colors, say 1 and 5, such that $G(1,5)$ contains no edge in $E$.
By Observation~\ref{obs:1stset}, we have $5 \notin C(S)$ and $1 \notin C(T)$.
By Observation~\ref{obs:2ndset},
every component in $G[S]$ (resp., $G[T]$) has color~1 (resp.,~5)
and moreover, 
we further make the number of vertices with color~1 (resp., color~5) in $S$ (resp., $T$) be as large as possible.
Hence, for every component $H$ in $G[S]$ or $G[T]$, one of the following holds:
\begin{itemize}
\item[(a)] If $H$ is a complete bipartite graph, then one of partite sets of $H$ has only color~1 or 5
(since it is an independent set).
\item[(b)] The number of vertices with color~1 or 5 in $H$ is at least $\frac{1}{3}|V(H)|$
(since a maximal independent set of every component has size at least $\frac{1}{3}|V(H)|$).
\item[(c)] Every vertex of $H$ is adjacent to a vertex of $H$ with color~1 (resp., color~5)
if $H \subseteq G[S]$ (resp., $H \subseteq G[T]$)
(since otherwise we can color a vertex not satisfying this condition by color~1 or~5,
which contradicts the maximality of the number of vertices with color~1 and~5).
\end{itemize}

Let $V_1$ and $V_5$ be the set of vertices with colors~1 and~5, respectively,
and let $G' = G - V_5$ and $G'' = G' - V_1$. 
Note that $G''$ is colored by three colors 2, 3 and 4.
Moreover,
$G''$ has at most three independent edges belonging to $G''[S]$ or $G''[T]$ by conditions (a), (b) and (c);
if a component $D$ of $G[S]$ or $G[T]$ containing an edge in $G''[S]$ or $G''[T]$,
then $D$ is a path or a cycle with at least $3q$ edges,
where $q = |E(D\cap G''[S])| + |E(D\cap G''[T])|$,
and hence, we have $|E(G''[S])| + |E(G''[T])| \leq 3$ by $|E| < 10$.

We first show the following claim.

\begin{cl}\label{cl:5c_cl1}
If $|E(G''[S])| + |E(G''[T])| \leq 2$,
or either $G''[S]$ or $G''[T]$ has no edge,
then $Kc(G,5)=1$.
\end{cl}
\begin{proof}[Proof of Claim~$\ref{cl:5c_cl1}$]
Suppose that $E(G''[S]) \neq \emptyset$.
Note that one of $(2,3)$-, $(2,4)$- and $(3,4)$-edges does not exist in $E(G''[S]) \cup E(G''[T])$
by $|E(G''[S])| + |E(G''[T])| \leq 2$.
If $E(G''[T]) = \emptyset$,
then we can easily have $C(S)=\{1,2,3\}$ and $C(T) = \{4,5\}$ by Lemma~\ref{lem:no_edge}
and hence $Kc(G,5)=1$ by Lemma~\ref{lem:no_intersect}.
If $|E(G''[S])|=|E(G''[T])| =1$,
then at least one of $G[S]$ or $G[T]$, say $G[S]$, contains only bipartite components
since every component is not a triangle.
Thus, as in the proof of Case~1 in Claim~\ref{cl:thm_1},
we can obtain $|C(S)|=2$ by K-changes,
and hence, $Kc(G,5)=1$ by Lemma~\ref{lem:no_intersect} since $|C(T)|=3$ by $|E(G''[T])| =1$.

So we may suppose that $G''[T]$ has exactly three independent edges, that is, $E_S = \emptyset$.
In this case,
since all vertices in $S$ are colored by~1 in $G$, 
every 5-colorings of $G$ are K-equivalent by Lemma~\ref{lem:no_intersect}.
\end{proof}

\begin{rem}\label{rem:5-3}
The proof of Claim~\ref{cl:5c_cl1} is shorten by avoiding a traingle.
If we allow a triangle,
then the proof of the case when $|E(G''[S])| = |E(G''[T])| = 1$ 
is as long as the proof of the later of Theorem~\ref{thm:c3e5} 
(i.e., Claims~\ref{cl:thm_2} and~\ref{cl:thm_3}).
\end{rem}

By Claim~\ref{cl:5c_cl1},
we may suppose that there are three edges $ab,uv,xy$ such that $ab \in E(G''[S])$ and $uv,xy \in E(G''[T])$,
and hence, by conditions (a), (b) and (c) and the assumption of the proposition,
$G''[S]$ and $G''[T]$ has no isolated vertex; otherwise, we have a contradiction to $|E| < 10$.
Next we show the Kempe equivalence of 3-colorings of $G''$.

\begin{cl}\label{cl:5c_cl2}
Every two 3-colorings of $G''$ with using colors $2,3,4$ are K-equivalent.
\end{cl}
\begin{proof}[Proof of Claim~$\ref{cl:5c_cl2}$]
We may suppose that there are two distinct 3-colorings $g_1$ and $g_2$ 
and that $g_i(a)=2$ and $g_i(b)=3$ for $i=1,2$
since we can make $C(\{a,b\}) = \{2,3\}$ by a $(2,4)$- or $(3,4)$-change in $G''$ by $|E(G''[S])|=1$.
Suppose $G$ is colored by the coloring $g_1$,
and then we show that $g_1$ can be transformed into $g_2$ by K-changes.
Note that it suffices to consider only the colors of $u$ and $v$ by the symmetry of $uv$ and $xy$.

\smallskip
\noindent
Case 1. $g_1(u)\neq g_2(u)$ and $g_1(v) = g_2(v)$.

In this case,
we are immediately done by applying a $(g_1(u),g_2(u))$-change 
to a $(g_1(u),g_2(u))$-component $H$ in $G''$ containing $u$,
since $V(H) = \{u\}$ by that for any $p \in \{a,b,v\}$, 
$g_1(p) = g_2(p)$ and $up \notin E(G'')$ if $g_2(u) = g_2(p)$.
(Recall that $uv$ and $xy$ are independent.)

\smallskip
\noindent
Case 2. $g_1(u)\neq g_2(u)$ and $g_1(v) \neq g_2(v)$.

Let $H$ be a $(g_1(u),g_2(u))$-component in $G''$ with $u$.
If $H$ contains both $a$ and $b$, that is, $g_1(u)=g_1(a)=2$ and $g_2(u)=g_1(b)=3$ by symmetry,
then $u$ is adjacent to neither $a$ nor $b$, 
that is, $g_1(v)=g_2(u)=3$ and $va \in E(G'')$.
In this case, since $g_2(v)=4$ ($g_2(v) \neq 2$ by $va \in E(G'')$ and $g_2(a)=2$), 
we can change the color of $v$ to color~4 by Lemma~\ref{lem:no_edge}(i).
Therefore, we can reduce the case to Case~1, which completes the proof of the claim.
\end{proof}

By Claim~\ref{cl:5c_cl2},
every 5-coloring of $G$ can be transformed into a 5-coloring $f$ such that
$f(a)=2,f(b)=3$,
$C(S)=\{1,2,3\}$, 
$C(T) \subseteq \{2,3,4,5\}$,
and all isolated vertices in $G[S]$ (resp., $G[T]$) have color~1 (resp.,~5).

Now we prove that $\Kc(G,5)=1$.
By the assumption,
each of $a,b$ (resp., each of $u,v,x,y$) is adjacent to exactly one vertex with color~1 (resp.,~5)
for example, if $a$ is adjacent to at least two vertices with color~1,
then the component with $a$ must be a complete bipartite graph since the degree of $a$ is at least~3,
that is, no edge of the component remains in $G''$ by (a), a contradiction.
Thus, let $c,c',w,w',z,z'$ be such vertices as shown in Figure~\ref{fig:g5_s},
where $f(c)=f(c')=1$ and $f(w)=f(w')=f(z)=f(z')=5$.
Note that $w,w',z,z'$ possibly coincide, e.g., $w = z', w' = z$ (forming a 6-cycle $uvzxyw$),
but the remaining proof will work for every variation of identification of those vertices.
(Observe that $c \neq c'$ since every component is not a triangle.)

\begin{figure}[htb]
\centering
\input{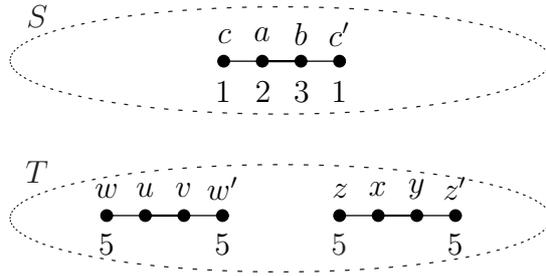}
\caption{The 5-coloring $f$ of $G$ when $ab \in E(G''[S])$ and $uv,xy \in E(G''[T])$}
\label{fig:g5_s}
\end{figure}

By Claim~\ref{cl:5c_cl1},
if we can transform $f$ into a 5-coloring $g$ such that $|E(G''[S])| + |E(G''[T])| \leq 2$
(where $V_1$ and $V_5$ are here determined by $g$),
then we are done.
In what follows, 
we show that $f$ can be transformed into such a 5-coloring $g$ by K-changes.

If $c$ or $c'$, say $c$, is not adjacent to a vertex with color~2 or~3,
then we can change $f(c)$ to~2 or~3 and then $f(a)$ to~1 by Lemma~\ref{lem:no_edge}(ii).
This operation is denoted by $\Tc_{c}$ for $c$.
Moreover,
if a vertex $p \in N_G(c) \cap T$ with color~2 or~3 is not adjacent to a vertex of color~4 or~5,
then $f(p)$ can be changed to~4 or~5 by Lemma~\ref{lem:no_edge}(ii)
and after that we can apply $\Tc_{c}$.
Thus, we may assume that $c$ is adjacent to $u,x$, $f(u) = 2$, $f(v)=4$, $f(x) = 3$ and $f(y)=4$.
However, since colors~4 and~5 can be freely exchanged in $T$,
we have a desired configuration that $u$ is not adjacent to a vertex of color~4,
by exchanging the colors of $v$ and $w'$.
Therefore, 
we can transform $f$ into a 5-coloring $g$ such that $|E(G''[S])| + |E(G''[T])| \leq 2$ by K-changes.
This completes the proof of the proposition.
\end{proof}

We are ready to prove Theorem~$\ref{thm:main}$.

\begin{proof}[Proof of Theorem~$\ref{thm:main}$]
Let $G$ be a $(k-1)$-colorable $B+E_{\ell}$ graph with $k \geq 4$ and $\ell < \binom{k}{2}$.
Suppose that $G$ is colored by $k$ colors $1,2,\dots,k$.
Since $|E| < \binom{k}{2}$,
there are two colors, say 1 and $k$, such that $G(1,k)$ contains no edge in $E$.
By Observation~\ref{obs:1stset}, we have $k \notin C(S)$ and $1 \notin C(T)$.
Moreover, by Observation~\ref{obs:2ndset},
each component in $G[S]$ (resp., $G[T]$) has color~1 (resp.,~$k$)
and we further make the number of vertices with color~1 (resp., color~$k$) in $S$ (resp., $T$) be as large as possible
by Lemma~\ref{lem:no_edge}(ii).
Thus, similarly to Proposition~\ref{prop:5colors},
we have the following for every component $H$ in $G[S]$ or $G[T]$:
\begin{itemize}
\item[(a)] If $H$ is a complete bipartite graph, 
then one of partite sets of $H$ has only color~1 or $k$.
\item[(b)] If $H$ is a path or a cycle,
then the number of vertices with color~1 or $k$ in $H$ is at least $\frac{1}{3}|V(H)|$.
\item[(c)] Every vertex of $H$ is adjacent to a vertex of $H$ with color~1 (resp., color~$k$)
if $H \subseteq G[S]$ (resp., $H \subseteq G[T]$).
\end{itemize}

Let $V_i$ be the vertices with color~$i$, and suppose that $|E_T| \geq \frac{1}{2}|E|$.
Observe that if we delete all vertices in $V_k$,
then for each component $H$ in $G[T]$, at least $\frac{2}{3}|E(H)|$ edges are deleted 
by the conditions (a), (b) and (c).
(If $H$ is a path, then at least $\frac{1}{3}(|V(H)| - 2q)$ vertices in $V_k$ have degree~2 in the path,
where $0 \leq q \leq 2$ denotes the number of end-vertices of $H$ in $V_k$.
In this case, we can delete 
$q + \frac{2}{3}(|V(H)| - 2q) = \frac{2}{3}|E(H)| - \frac{1}{3}q  + \frac{2}{3} \geq \frac{2}{3}|E(H)|$ edges
since $|E(H)|=|V(H)|-1$.)
Thus, for the resulting graph $G' = G-V_k$,
$$
|E(G')| \leq |E| - \frac{2}{3} \cdot \frac{1}{2}|E| = \frac{2}{3}|E| < \frac{2}{3} \cdot \binom{k}{2},
$$
and hence, if $|E(G')| \geq \binom{k-1}{2}$, we have $k < 6$.
This implies that for $k \geq 6$, we have $|E(G')| < \binom{k-1}{2}$ by the above deletion.
Moreover, if $k \geq 7$,
$G'$ is 5-colorable (that is, $(k-2)$-colorable) 
since $G$ is 6-colorable by the assumption of the theorem.
Therefore, $G'$ satisfies all assumptions of the theorem when $k \geq 7$,
we can prove the theorem by induction on $k$ for $k \geq 7$ and Lemma~\ref{lem:no_intersect}
(it is trivial that each component of $G'[S]$ and $G'[T]$ is either a path, 
a cycle of length at least~4 or a complete bipartite graph).

Hence we suffice to show the theorem for $k \in \{4,5,6\}$.
For $k=4$ (resp., $k=5$), we are done by Theorem~\ref{thm:c3e5} (resp., Proposition~\ref{prop:5colors}).
So we may assume $k=6$.
Let $G''$ be the graph obtained from $G'$ by deleting all vertices in $V_1$.
Similarly to the previous argument, we have 
$$
|E(G'')| \leq |E|-\frac{2}{3}|E| = \frac{1}{3}|E| < \frac{1}{3} \cdot \binom{6}{2} = 5.
$$
Moreover, each component in $G''[S]$ and $G''[T]$ is an edge or an isolated vertex
(otherwise, contradicts (a), (b), (c) or the assumption for components induced by $E$).
Thus, $\Kc(G'',4)=1$ by Theorem~\ref{thm:bm5} since $G''$ is clearly 4-colorable.
Therefore, since we can change a 6-coloring of $G$ so that $C(S) \cap C(T) = \emptyset$, 
we have $\Kc(G,6) = 1$ by Lemma~\ref{lem:no_intersect}.
The proof of the theorem is completed.
\end{proof}

\section{Conclusion}\label{sec:7}

The Kempe equivalence of colorings of graphs is an important concept in many ways.
Since every two colorings of any bipartite graph are Kempe equivalent,
it is a natural problem 
whether given two colorings of an {\it almost} bipartite graph are Kempe equivalent or not.
So we consider the Kempe equivalence of $B+E_{\ell}$ graphs,
and in particular, 
we propose a conjecture of interest stating that 
every two $k$-colorings of any $(k-1)$-colorable $B+E_{\ell}$ graph with $k \geq 4$ and $\ell < \binom{k}{2}$
are Kempe equivalent (Conjecture~\ref{conj:main}).
This conjecture is more lately disproved for $k \geq 8$ by Cranston and Feghali~\cite{cranston2023kempe},
but the conjecture is possibly true for small $k$.
Throughout this paper,
we show several non-trivial partial solutions for small $k$ and the sharpness of the conjecture.
In particular,
the assumption of Theorem~\ref{thm:main} needs to prove the theorem by induction on $k$.
To prove Conjecture~\ref{conj:main} for $k \in \{6,7\}$,
we probably need to find a new proof method without induction on $k$.
Moreover, we also conjecture that every two 4-colorings of any 4-critical graph are Kempe equivalent
(Conjecture~\ref{conj:4cri}),
and our results are also partial solutions of this conjecture.
However, each of our proofs does not use the 4-criticality of graphs.
Therefore, 
we strongly believe that 
we can prove the Kempe equivalence of 4-colorings of 4-critical graphs as the planar case
by combining the 4-criticality and several ideas that we have not found yet.

\bibliographystyle{alpha}
\bibliography{main}

\end{document}